\documentclass[11pt]{article}
    \usepackage[T1]{fontenc}
	\usepackage[utf8]{inputenc}
	\usepackage[english]{babel}
	\usepackage{amsfonts}
	\usepackage{amssymb}
	\usepackage{enumitem}
	\usepackage{bbm}
	\usepackage{mathtools}
	\usepackage{subfigure}
	\usepackage{booktabs}
	\usepackage{graphicx}
	
	\usepackage{soul}
	\usepackage{xspace}

    \usepackage{titlesec}

    \titleformat{\section}
    {\centering\normalfont\scshape}{\thesection}{1em}{}
    
    \titleformat{\subsection}
    {\centering\normalfont\itshape}{\thesubsection}{1em}{}
	
	\usepackage{setspace}
	\usepackage{natbib}
	\usepackage{bibunits}
	\defaultbibliography{Bibliography.bib}  
	\defaultbibliographystyle{apalike}

	\usepackage{hyperref}
	\usepackage{cleveref}

	\usepackage[framemethod=tikz]{mdframed}
	\usepackage{todonotes}
	\presetkeys{todonotes}{color=black!5,inline}{} 

	\newcommand{\rf}[1]{\comment{Reference: \url{#1}}}

    \usepackage{tcolorbox}
    \newtcbox{\feedback}{nobeforeafter,colframe=black,colback=white,boxrule=0.5pt,arc=2pt,
      boxsep=0pt,left=2pt,right=2pt,top=2pt,bottom=2pt,tcbox raise base}
      
    \usepackage{amsthm}
     
    \newtheorem{thm}{Theorem}

    \newtheorem{prop}{Proposition}
    \newtheorem{lem}{Lemma}

\newcommand\blfootnote[1]{%
  \begingroup
  \renewcommand\thefootnote{}\footnote{#1}%
  \addtocounter{footnote}{-1}%
  \endgroup
}

	\renewcommand{\P}{\mathop{}\!\textnormal{P}}
	\newcommand{\E}{\mathop{}\!\textnormal{E}}

	\newcommand{\N}{\mathcal{N}}

	\newcommand{\0}{\mathbf{0}}
	\newcommand{\1}{\mathbf{1}}
	\newcommand{\I}{\mathbb{I}}
	\renewcommand{\O}{\mathbb{O}}

	\newcommand{\R}{\mathbb{R}}

	\newcommand{\OLS}{\textnormal{OLS}}  
	
\title{
\Large
Bias Reduction in Instrumental Variable Estimation through First-Stage Shrinkage}
\author{Jann Spiess}
\date{\textsc{Working Paper} \\ This version: October 31, 2017}

\begin{document}
    
    \maketitle
    
    \begin{abstract}
		The two-stage least-squares (2SLS) estimator is known to be biased when its first-stage fit is poor. I show that better first-stage prediction can alleviate this bias. In a two-stage linear regression model with Normal noise, I consider shrinkage in the estimation of the first-stage instrumental variable coefficients. For at least four instrumental variables and a single endogenous regressor, I establish that the standard 2SLS estimator is dominated with respect to bias. The dominating IV estimator applies James–Stein type shrinkage in a first-stage high-dimensional Normal-means problem followed by a control-function approach in the second stage. It preserves invariances of the structural instrumental variable equations.
    \end{abstract}

    \blfootnote{
        \hspace{-\baselineskip}
        Jann Spiess, Department of Economics, Harvard University, \href{mailto:jspiess@fas.harvard.edu}{\texttt{jspiess@fas.harvard.edu}}.
        I thank Gary Chamberlain, Maximilian Kasy, and Jim Stock for helpful comments.
    }%
    
    \section*{Introduction}

    The standard two-stage least-squares  (2SLS) estimator is known to be biased towards the OLS estimator when instruments are many or weak.
    In a linear instrumental variables model with one endogenous regressor, at least four instruments, and Normal noise,
    I propose an estimator that combines James--Stein shrinkage in a first stage with a second-stage control-function approach.
    Unlike other IV estimators based on James--Stein shrinkage, my estimator reduces bias uniformly relative to 2SLS. Unlike LIML, it is invariant with respect to the structural form and translation of the target parameter.
    
    I consider the first stage of a two-stage least-squares estimator as a high-dimensional prediction problem, to which I apply rotation-invariant shrinkage akin to \cite{James:1992jm}.
    Regressing the outcome on the resulting predicted values of the endogenous regressor directly would shrink the 2SLS estimator towards zero, which could increase or decrease bias depending on the true value of the target parameter.
    Conversely, shrinking the 2SLS estimator towards the OLS estimator can reduce risk \citep{hansen2017stein}, but increases bias towards OLS.
    Instead, my proposed estimator uses the first-stage residuals as controls in the second-stage regression of the outcome on the endogenous regressor.
    If no shrinkage is applied, the 2SLS estimator is obtained as a special case, while a variant of \cite{James:1992jm} shrinkage that never fully shrinks to zero uniformly reduces bias.

    The proposed estimator is invariant to a group of transformations that include translation in the target parameter.
    While the limited-information maximum likelihood estimator (LIML) can be motivated rigorously as an invariant Bayes solution to a decision problem \citep{Chamberlain:2007uz}, these transformations rotate the (appropriately re-parametrized) target parameter and invariance applies to a loss function that has a non-standard form in the original parametrization.
    In particular, unlike LIML, the invariance of my estimator applies to squared-error loss.
    
    The two-stage linear model is set up in Section~\ref{sect:setup}.
    Section~\ref{sect:iv} proposes the estimator and establishes bias improvement relative to 2SLS. Section~\ref{sect:invariance} develops invariance properties of the proposed estimator.

    \section{Two-Stage Linear Regression Setup}
    \label{sect:setup}
    
    I consider estimation of the structural parameter $\beta \in \R$ in
    the standard two-stage linear regression model
    \begin{align}
    \begin{aligned}
    \label{eqn:normal}
        Y_i &= \alpha + X_i' \beta + W_i' \gamma + U_i \\
        X_i &= \alpha_X + Z'_i \pi + W'_i \gamma_X + V_i
    \end{aligned}
    \end{align}
    from $n$ iid observations $(Y_i,X_i,Z_i,W_i)$, where $X_i \in \R$ is the regressor of interest (assumed univariate), $W_i \in \R^k$ control variables, $Z_i \in \R^\ell$ instrumental variables, and $(U_i,V_i)' \in \R^2$ is homoscedastic (wrt $Z_i$), Normal noise.
    $\alpha$ is an intercept,%
    \footnote{We could alternatively include a constant regressor in $X_i$ and subsume $\alpha$ in $\beta$. I choose to treat $\alpha$ separately since I will focus on the loss in estimating $\beta$, ignoring the performance in recovering the intercept $\alpha$.} 
    and $\gamma$ and $\pi$ are nuisance parameters.
    This model could be motivated by a latent variable present in both outcome and first-stage equation under appropriate exclusion restrictions as in \cite{Chamberlain:2007uz}.\footnote{In this section, the intercepts $\alpha, \alpha_X$ could be subsumed in the control coefficients $\gamma,\gamma_X$ without loss, but I maintain this notation to keep it consistent.}
    
    Throughout this document, I write upper-case letters for random variables (such as $Y_i$) and lower-case letters for fixed values (such as when I condition on $X_i = x_i$).
    When I suppress indices, I refer to the associated vector or matrix of observations, e.g. $Y \in \R^n$ is the vector of outcome variables $Y_i$ and $X \in \R^{n \times m}$ is the matrix with rows $X'_i$.
    
        For the noise I use the notation
    \begin{align*}
        \begin{pmatrix}
            U_i \\ V_i
        \end{pmatrix}
        | Z_i = z_i, W_i = w_i
        \sim \N\left(
            \0_2, \begin{pmatrix}
            \sigma^2 & \rho \sigma \tau \\ 
            \rho \sigma \tau  & \tau^2
        \end{pmatrix}
        \right)
    \end{align*}
    for some $\rho \in (-1,1)$.
    The reduced form is
    \begin{align*}
        \begin{pmatrix}
            Y \\
            X
        \end{pmatrix}
        |Z = z, W{=}w
        \sim
        \N
        \left(
            \begin{pmatrix}
                \alpha + z \pi_Y + w \gamma_Y  \\
                \alpha_X + z \pi + w \gamma_X
            \end{pmatrix},
            \Sigma \otimes \I_{2n}
        \right)
    \end{align*}
    with
    \begin{align*}
        &
        \begin{matrix}
            \pi_Y = \pi \beta, \\
            \gamma_Y = \gamma + \gamma_X \beta,
        \end{matrix}
        &
        \Sigma 
        &=  \begin{pmatrix}
                 \sigma^2 + 2 \rho \beta \sigma \tau + \beta^2 \tau^2 & \rho \sigma \tau + \beta \tau^2 \\
                 \rho \sigma \tau + \beta \tau^2 & \tau^2
            \end{pmatrix}.
    \end{align*}
    Note that there is a one-two-one mapping between reduced-form and structural-form parameters provided that the proportionality restriction $\pi_Y = \pi \beta$ holds.
    I develop a natural many-means form directly from the structural model, which is thus without loss, but not without consequence.
    Throughout, our interest will be in estimating $\beta$ for many instruments (large $\ell$).

    We have
    $
        U_i | X_i=x_i,Z_i=z_i,W_i=w_i
        \sim \N\left(\frac{\rho \sigma}{\tau} v_i,(1-\rho^2) \sigma^2\right)
    $
    where $v_i = x_i - \alpha_X - z_i'\pi  - w_i'\gamma_X$.
    Given $w \in \R^{n \times k}$ and $z \in \R^{n \times \ell}$,
    where I assume that $(\1,w,z)$ has full rank $1 + k + \ell \leq n - 1$,
    let $q = (q_{\1},q_w,q_z,q_r) \in R^{n \times n}$ orthonormal
    where $q_{\1} \in \R^n, q_w \in \R^{n \times k}, q_z \in \R^{n \times \ell}$
    such that $\1$ is in the linear subspace of $\R^n$ spanned by $q_{\1} \in \R^{n}$
    (that is, $q_{\1} \in \{\1/n,-\1/n\}$),
    the columns of $(\1,w)$ are in the space spanned by the columns of $(q_{\1},q_w)$,
    and the columns of $(\1,w,z)$ are in the space spanned by the columns of $(q_{\1},q_w,q_z)$.
    (As above, such a basis exists, for example, by an iterated singular value decomposition.)
    Then,
    \begin{align*}
        \begin{pmatrix}
            q'_z X \\
            q'_r X
        \end{pmatrix}
        | Z{=}z,W{=}w
        &\sim \N\left(
            \begin{pmatrix}
            q'_z z \pi \\
            \0_{n - 1 - k - \ell}
            \end{pmatrix},
            \tau^2 \I_{n^*}
        \right)
        \\
        \begin{pmatrix}
            q'_z Y \\
            q'_r Y
        \end{pmatrix}
        | X{=}x,Z{=}z,W{=}w
        &\sim \N\left(
            \begin{pmatrix}
                q'_z x \\
                q'_r x
            \end{pmatrix}
            \beta
            +
            \begin{pmatrix}
            q'_z x {-} q'_z z \pi \\
            q'_r x
            \end{pmatrix}
            \frac{\rho \sigma}{\tau},
            (1 {-} \rho^2) \sigma^2 \I_{n^*}
        \right),
    \end{align*}
    where $n^* = n - 1 - k$.
    Writing $X^*_z, X^*_r,Y^*_z, Y^*_r$ for the respective subvectors,
    \begin{align*}
        X^* &= \begin{pmatrix}
            X^*_z \\
            X^*_r
        \end{pmatrix},
        &
        Y^* &= \begin{pmatrix}
            Y^*_z \\
            Y^*_r
        \end{pmatrix},
    \end{align*}
    $\mu = q'_z z \pi$, and $s= n - 1 - k - \ell$,
    we arrive at the canonical structural form
    \begin{align}
    \label{eqn:caniv}
    \begin{split}
        X^*
        &\sim \N\left(
            \begin{pmatrix}
            \mu \\
            \0_{s}
            \end{pmatrix},
            \tau^2 \I_{\ell + s}
        \right)
        \\
        Y^*
        |
        X^* {=} x^*
        &\sim \N\left(
            x^*
            \beta
            +
            \left(x^* - 
            \begin{pmatrix}
                \mu\\
                \0_s
            \end{pmatrix}
            \right)
            \frac{\rho \sigma}{\tau},
            (1 - \rho^2) \sigma^2 \I_{\ell + s}
        \right),
    \end{split}
    \end{align}
    where I have suppressed conditioning on $Z{=}z,W{=}w$ (and omit it from here on).

    \section{Control-Function Shrinkage Estimator}
    \label{sect:iv}

    Given an estimator $\hat{\mu} = \hat{\mu}(X^*)$ of $\mu$,
    a feasible implied estimator for $\beta$ in \autoref{eqn:caniv} is the coefficient 
    on $X^*$ in a linear regression of $Y^*$ on $X^*$ and the control function
    $X^* - (\hat{\mu}',\0'_s)'$.
    (The two-stage least-squares estimator $\hat{\beta}^{\textnormal{2SLS}} = \frac{(Y^*_z)'X^*_z}{(X^*_z)'X^*_z}$
    is obtained from the first-stage OLS solution $\hat{\mu}^{\OLS} = X^*_z$.
    It is biased towards the OLS estimator $\hat{\beta}^{\OLS} = \frac{(Y^*)'X^*}{(X^*)'X^*}$.)

    For high-dimensional $\mu$,
    a natural estimator for $\mu$ is a shrinkage estimator of the form
    $
        \hat{\mu}(X^*) = c(X^*) X^*_z
    $
    with scalar $c(X^*)$.
    The conditional bias of the implied control-function estimator $\hat{\beta}$ takes a particularly simple form for this class of estimators:
    
    \begin{lem}[Conditional bias of CF--shrinkage estimators]
    \label{lem:bias}
        For $x^* \in \R^{\ell + s}$ with $c(x^*) \neq 0$,
        \begin{align*}
            \E[\hat{\beta} | X^*=x^*] - \beta
            &= \E\left[\frac{\hat{\mu}'(\hat{\mu} - \mu)}{\hat{\mu}'\hat{\mu}}\middle| X^*=x^*\right] \frac{\rho \sigma}{\tau} \\
            &= \left(1 - \frac{1}{c(x^*)} \frac{(x_z^*)'\mu}{(x_z^*)' x_z^*}\right) \frac{\rho \sigma}{\tau}.
        \end{align*}    
    \end{lem}
    
    Shrinkage in the \cite{James:1992jm} estimator (for unknown $\tau^2$) takes the form $c(x^*) = 1 - p \frac{\|x^*_r \|^2}{\| x^*_z \|^2}$.
    This shrinkage pattern (and its positive-part variant) is unappealing here, as it can cross zero, around which point the estimator diverges.
    A natural variant that mitigates this problem is
    \begin{align*}
        c(x^*) = \frac{1}{1 + p \frac{\|x^*_r \|^2}{\| x^*_z \|^2}}
        = \frac{\| x^*_z \|^2}{\| x^*_z \|^2 + p \|x^*_r \|^2},
    \end{align*}
    which behaves as $1 - p \frac{\|x^*_r \|^2}{\| x^*_z \|^2}$
    for small $p \frac{\|x^*_r \|^2}{\| x^*_z \|^2}$, but never quite reaches zero.
    
    \begin{thm}[Bias dominance through shrinkage]
        \label{thm:biasdom}
        Assume that $\ell \geq 4$ and $p \in \left(0,2 \frac{\ell - 2}{s}\right)$.
        Then
        $
            |\E[\hat{\beta} | Z{=}z,W{=}w] - \beta|
            < |\E[\hat{\beta}^{\textnormal{2SLS}} | Z{=}z,W{=}w] - \beta|
        $
        provided $\rho \neq 0$ and $\| \mu \| \neq 0$ (otherwise equality).
    \end{thm}
    
    The requirement $\ell \geq 4$ is an artifact of this specific shrinkage pattern and dominance should extend to $\ell = 3$ for an appropriate modification. 
    
    \begin{proof}
        For the (rescaled) bias, where $\lambda = p s$ and $M = X^*_z$, we have by \autoref{lem:bias} that
        \begin{align*}
            B(\lambda)
            &= 
            \frac{\tau}{\rho \sigma} \E[\hat{\beta} - \beta] 
            = \E\left[1 - \frac{\| X^*_z \|^2 + p \|X^*_r \|^2}{\| X^*_z \|^2} \frac{(X_z^*)'\mu}{(X_z^*)' X_z^*}\right] \\
            &= \E\left[  1 - \frac{\| X^*_z \|^2 + p \E\left[\|X^*_r \|^2 \middle|X^*_z \right]}{\| X^*_z \|^2} \frac{(X_z^*)'\mu}{(X_z^*)' X_z^*}\right] \\
            &=
            \E\left[
                1 - \frac{M' \mu}{\|M\|^2} - \lambda \tau^2 \frac{M' \mu}{\|M\|^4}
            \right],
        \end{align*}
        provided that $\E\left|
                1 - \frac{M' \mu}{\|M\|^2} - \lambda \tau^2 \frac{M' \mu}{\|M\|^4}
            \right| < \infty$.
        By the multi-dimensional version of \citeauthor{Stein:1981ie}'s (\citeyear{Stein:1981ie}) lemma for $h(M) = \frac{1}{\|M\|^2}$,
        \begin{align*}
            - 2 \tau^2 \E\left[ \frac{M}{\|M\|^4} \right]
            = 
            \tau^2 \E\left[ \nabla h(M) \right]
            =
            \E\left[ (M - \mu) h(M) \right]
            =
            \E\left[ \frac{M - \mu}{\|M\|^2} \right],
        \end{align*}
        again provided that all moments exist.

        For the existence of moments, 
        note that by Cauchy--Schwarz and Jensen it suffices to consider $\E\left\| M / \|M\|^4 \right\| = \E[\|M\|^{-3}]$.
        To establish that this expectation is finite,
        note that the distribution of $\|M\|^2 / \tau^2$, a non-central $\chi^2$ distribution with $\ell$ degrees of freedom and non-centrality parameter $\|\mu\|^2/\tau^2$, is first-order stochastically dominating a central $\chi^2$ distribution with $\ell$ degrees of freedom, so it is sufficient to establish $\E[(X^2)^{-3/2}] < \infty$ where $X^2$ has a central $\chi^2$ distribution with $\ell$ degrees of freedom.
        Now, the density $f(y)$ of $(X^2)^{3/2}$ is proportional to $y^{\ell/3 - 1} \exp(-y^{2/3}/2)$,
        implying $\lim_{y \searrow 0} f(y) / y^\alpha = 0$ for $\ell \geq 4$ and, say, $\alpha = 1/4 > 0$.
        The existence of the inverse moment, i.e. $E[(X^2)^{-3/2}] < \infty$, follows by \cite{Piegorsch:1985ge}.

        We thus have
        \begin{align*}
            \E\left[ \frac{M' \mu}{\|M\|^4} \right]
            =
            \frac{-1}{2 \tau^2}
            \E\left[ \frac{(M - \mu)' \mu}{\|M\|^2} \right],
        \end{align*}
        which yields
        \begin{align*}
            B(\lambda)
           &=
            \E\left[
                1 - \frac{M' \mu}{\|M\|^2} + \frac{\lambda}{2} \frac{(M - \mu)' \mu}{\|M\|^2}
            \right]
            \\
            &=
            \E\left[
                1 - \frac{\|M\|^2 - (M - \mu)' M}{\|M\|^2} + \frac{\lambda}{2} \frac{\|M^2\| - \|\mu\|^2 - (M - \mu)' M
                }{\|M\|^2}
            \right]
            \\
            &=
            \frac{\lambda}{2}
            - \frac{\lambda}{2}
            \E\left[\frac{\|\mu\|^2}{\| M \|^2}\right]
            - \frac{\lambda - 2}{2}
            \E\left[\frac{(M - \mu)' M}{\|M \|^2}\right].
        \end{align*}
        
        Denote by $K$ a Poisson random variable with mean $\kappa = \frac{\| \mu \|^2}{2 \tau^2} > 0$. ($B(\lambda)$ is constant at $1$ for $\| \mu \| = 0$, and there remains nothing to show.)
        From  \citet[(9), (16)]{James:1992jm} we have that
        \begin{align*}
            \E\left[\frac{\|\mu\|^2}{\| M \|^2}\right]
            &=  \E\left[\frac{2 \kappa}{\ell - 2 + 2 K}\right] = Q(\ell),
            \\
            \E\left[\frac{(M - \mu)' M}{\|M \|^2}\right]
            &=  \E\left[\frac{\ell-2}{\ell - 2 + 2 K}\right] = P(\ell).
        \end{align*}
        It immediately follows from
        \begin{align*}
            B(\lambda) = P(\ell) - \frac{\lambda}{2} (P(\ell) + Q(\ell) - 1)
        \end{align*}
        that the bias for the unshrunk reference estimator ($\lambda = 0$, 2SLS) is
        $B(0) = P(\ell) > 0$,
        and that $B(\lambda)$ is decreasing in $\lambda$ since $P(\ell) + Q(\ell) \geq 1$ by Jensen's inequality (with strict inequality unless $\| \mu \| = 0$).
        The (infeasible) bias-minimizing choice of $\lambda$ is given by
        \begin{align*}
            \lambda^*
            =
            \frac{2 P(\ell)}{P(\ell) + Q(\ell) - 1}
            = \frac{\ell-2}{\frac{\ell-2}{2} + \kappa - 1 / \E\left[(\frac{\ell-2}{2} + K)^{-1}\right]}.
        \end{align*}
        To conclude the proof, I assert (and prove below) that, for any $a \geq 1$,
        \begin{align}
            \label{eqn:poissoninequality}
            \E\left[(a + K)^{-1}\right] \leq \frac{1}{a + \nu - 1}.
        \end{align}
        With $a = \frac{\ell - 2}{2}$ it follows that $\frac{\ell-2}{2} + \kappa - 1 / \E\left[(\frac{\ell-2}{2} + K)^{-1}\right] \leq 1$ and thus $\lambda^* \geq \ell - 2$.
        We obtain $|B(\lambda)| \leq |B(0)|$ (dominance over 2SLS in terms of bias) for all $\lambda \in (0,\ell - 2)$ by strict monotonicity of $B(\lambda)$, which yields the theorem.

        To establish \autoref{eqn:poissoninequality},
        fix $a \in \R$ with $a \geq 1$
        and note that for $K$ Poisson with parameter $\nu$
        \begin{align*}
            \E\left[\frac{\nu}{a + K}\right]
            &=
            \sum_{\iota = 0}^\infty \frac{\nu}{a + \iota} \frac{\nu^{\iota} \exp(-\nu)}{\iota!} 
            = \sum_{\iota = 0}^\infty \frac{\iota + 1}{a + \iota} \frac{\nu^{\iota + 1} \exp(-\nu)}{(\iota + 1)!} \\
            &= \sum_{\iota = 1}^\infty \frac{\iota}{a + \iota - 1} \frac{\nu^{\iota} \exp(-\nu)}{\iota!}.
        \end{align*}
        For $a = 1$, thus $\E\left[\frac{\nu}{a + K}\right] = 1 - \exp(-\nu) \leq 1$.
        For $a > 1$, 
        \begin{align*}
            \E\left[\frac{\nu}{a + K}\right]
            &= \sum_{\iota = 0}^\infty \frac{\iota}{a + \iota - 1} \frac{\nu^{\iota} \exp(-\nu)}{\iota!}
            = \E\left[\frac{K}{a + K - 1}\right]
            \leq \frac{\nu}{a + \nu - 1}
        \end{align*}
        by Jensen's inequality applied to the concave function $x \mapsto \frac{x}{a - 1 + x} (x \geq 0)$.
        In both cases, \autoref{eqn:poissoninequality} follows by dividing by $\nu$,
        yielding a generalization of an inequality in \citet[Theorem~6]{moser2008expectations} to non-integer $a$.
        \end{proof}

    \section{Invariance Properties}
    \label{sect:invariance}
        
        The estimator $\hat{\beta}$ developed in the previous section has invariance properties in a decision problem, where in spirit and notation I follow the treatment of LIML in \cite{Chamberlain:2007uz}.

        First I fix the sample and action spaces, as well as a class of loss functions, for the decision problem of estimating $\beta$. 
        Starting with \autoref{eqn:caniv},
        I write $\mathcal{Z} = (\R^{\ell + s})^2$ for the sample space from which $(X^*,Y^*)$ is drawn according to $\P_\theta$,
        where I parametrize $\theta = (\beta,\mu,\rho,\sigma,\tau) \in \Theta = \R \times \R^\ell \times \R^3_{\geq 0}$.
        The action space is $\mathcal{A}= \R$, from which an estimate of $\beta$ is chosen.
        I assume that the loss function $L: \Theta \times \mathcal{A} \rightarrow \R$ can be written as $L(\theta,a) = \ell(a - \beta)$ for some sufficiently well-behaved $\ell: \R \rightarrow \R$ (such as squared-error loss $L(\theta,a) = (a - \theta)^2$).
        The estimator $\hat{\beta}: \mathcal{Z} \rightarrow \mathcal{A}$ from the previous section is a feasible decision rule in this decision problem.
        
        For an element $g = (g_\beta,g_z,g_r)$ in the (product) group $G = \R \times O(\ell) \times O(s)$, where $\R$ denotes the group of real numbers with addition (neutral element $0$) and $O(\ell)$ the group of ortho-normal matrices in $\R^{\ell \times \ell}$ with matrix multiplication (neutral element $\I_\ell$),
        consider the following set of transformations (which are actions of $G$ on $\mathcal{Z}, \Theta, \mathcal{A}$):
        \begin{itemize}
            \item Sample space: $m_\mathcal{Z}: G \times \mathcal{Z} \rightarrow \mathcal{Z}$,
                            \begin{align*}
                                (g,(x^*,y^*))
                                \mapsto
                                \left(
                                    \begin{pmatrix}
                                        g_z & \O \\
                                        \O & g_r
                                    \end{pmatrix}
                                    x^*,
                                    \begin{pmatrix}
                                        g_z & \O \\
                                        \O & g_r
                                    \end{pmatrix}
                                    (y^*
                                    +
                                    g_\beta
                                    x^*)
                                \right)
                            \end{align*}
            \item Parameter space: $m_\Theta: G \times \Theta \rightarrow \Theta$,
                    \begin{align*}
                        (g,\theta) \mapsto
                        (\beta + g_\beta,g_z \mu,\rho,\sigma,\tau)
                    \end{align*}
            \item Action space: $m_\mathcal{A}: G \times \mathcal{A} \rightarrow \mathcal{A}, (g,a) \mapsto a+g_\beta$
        \end{itemize}

        These transformations are tied together by leaving model and loss invariant. Indeed, the following result is immediate from \autoref{eqn:caniv}:
        \begin{prop}[Invariance of model and loss]
            \leavevmode
            \begin{enumerate}
                \item The model is invariant: $m_\mathcal{Z}(g,(X^*,Y^*)) \sim \P_{m_\Theta(g,\theta)}$ for all $g \in G$.
                \item The loss is invariant: $L(m_\Theta(g,\theta),m_\mathcal{A}(g,a)) = L(\theta,a)$ for all $g \in G$.
            \end{enumerate}
        \end{prop}

        A decision rule $d: \mathcal{Z} \rightarrow \mathcal{A}$ is invariant if, for all $(g,(x^*,y^*)) \in G \times \mathcal{Z}$,
            $d(m_\mathcal{Z}(g,(x^*,y^*))) = m_\mathcal{A}(g,d((x^*,y^*)))$.
        The estimator $\hat{\beta}$ above is included in a class of invariant decision rules:
        \begin{prop}[Invariance of a class of control-function estimators]
            Consider a control-function decision rule $d((x^*,y^*))$ obtained as the coefficient on $x^*$ in a linear regression of $y^*$ on $x^*$, controlling for $x^* - (c(\|x^*_z\|,\|x^*_r\|) (x^*_z)',\0')'$,
            where $c(\|x^*_z\|,\|x^*_r\|)$ scalar (and measurable).
            Then $d$ is an invariant decision rule with respect to the above actions of $G$.
        \end{prop}

        \begin{proof}
            Fix $(x,y) \in \mathcal{Z}$ and consider $d((x,y))$.
            Note first that $c = c(\|x_z\|,\|x_r\|)$ is invariant to the action of $G$ on $\mathcal{Z}$.
            The decision rule is
            \begin{align*}
                d((x,y))
                = \frac{x' a(x) y}{x' a(x) x}
            \end{align*}
            where
            \begin{align*}
                a(x) = \I - b(x) (b(x)'b(x))^{-1} b(x)'
                \text{ for }
                b(x) =
                \begin{pmatrix}
                    (1-c) x_z \\ x_r
                \end{pmatrix}.
            \end{align*}
            Now for any $g \in G$, where I write $q_g = \begin{pmatrix} g_z & \O \\ \O & g_r \end{pmatrix}$, we have $b( q_g x) = q_g b(x)$ and thus $a(q_g x) = q_g a(x) q_g'$.
            It is immediate that
            \begin{align*}
                d(m_\mathcal{Z}(g,(x,y)))
                &= d((q_g x,q_g y + g_\beta q_g x))
                = \frac{x' a(x) y}{x' a(x) x} + g_\beta \frac{x' a(x) x}{x' a(x) x} \\
                &= d((x,y)) + g_\beta = m_\mathcal{A}(g,d((x,y))),
            \end{align*}
            as claimed.
        \end{proof}

    \section*{Conclusion}

    An application of James--Stein shrinkage to instrumental variables in a canonical structural transformation consistently reduces bias.
    The specific estimator is invariant to a group of transformations of the structural form that involves translation of the target parameter.
    
    In a companion paper \citep{JSC}, I show how analogous shrinkage in at least three control variables provides consistent loss improvement over the least-squares estimator without introducing bias, provided that treatment is assigned randomly.
    Together, these results suggests different roles of overfitting in instrumental variable and control coefficients, respectively: while overfitting to instrumental variables in the first stage of a two-stage least-squares procedure induces bias, overfitting to control variables induces variance.

    \bibliography{Bibliography}

\end{document}